\documentclass[11pt,reqno]{amsart}
\usepackage[latin1]{inputenc} 
\usepackage{indentfirst}
\usepackage{amsmath} 
\usepackage{amsfonts} 
\usepackage{amsthm} 
\usepackage{epsfig} 
\usepackage{graphicx} 
\usepackage{multirow} 
\usepackage{verbatim} 
\usepackage{rotating} 
\usepackage{lettrine}
\usepackage{lscape}
\usepackage{ae}
\usepackage{float}
\usepackage[all]{xy}
\usepackage{icomma} 
\usepackage{amsmath} 
\usepackage{amsfonts} 
\usepackage{amsthm} 
\usepackage{varioref} 
\usepackage{layout}
\usepackage[Lenny]{fncychap} 
\usepackage{enumerate}
\usepackage{amsfonts, amsmath, amssymb, stmaryrd, latexsym}
\usepackage{dsfont}
\usepackage{wrapfig, color, epsfig}
\usepackage{array}
\usepackage{longtable}
\usepackage{geometry}
\usepackage[frenchb,english]{babel}
\usepackage{mathrsfs}
\usepackage[mathscr]{euscript}

\newtheorem{thm}{Theorem}
\newtheorem{Def}{Definition}
\newtheorem{prop}{Proposition}
\newtheorem{cor}{Corollary}

\newtheorem{exs}{Examples}

\newtheorem{lem}{Lemma}
\newtheorem{rem}{Remark}
\newtheorem{rems}{Remarks}
\newcommand{\bexs}{\begin{exs}}
\newcommand{\eexs}{\end{exs}}
\newcommand{\bt}{\begin{thm}}
\newcommand{\et}{\end{thm}}
\newcommand{\bd}{\begin{Def}}
\newcommand{\ed}{\end{Def}}
\newcommand{\bl}{\begin{lem}}
\newcommand{\el}{\end{lem}}
\newcommand{\bp}{\begin{prop}}
\newcommand{\ep}{\end{prop}}
\newcommand{\brem}{\begin{rem}}
\newcommand{\erem}{\end{rem}}
\newcommand{\brems}{\begin{rems}}
\newcommand{\erems}{\end{rems}}
\newcommand{\bc}{\begin{cor}}
\newcommand{\ec}{\end{cor}}
\newcommand{\bpr}{\begin{proof}}
\newcommand{\epr}{\end{proof}}
\def\NN{\mathbb{N}}

\def\QQ{\mathbb{Q}}

\def\ZZ{\mathbb{Z}}

\begin{document}
\title[ On the rank of the $2$-class group...]{On the rank of the $2$-class\\ group of $\QQ(\sqrt{p}, \sqrt{q},\sqrt{-1} )$}
\author[Abdelmalek AZIZI]{Abdelmalek Azizi}
\address{Abdelmalek Azizi and Abdelkader Zekhnini: Département de Mathématiques, Faculté des Sciences, Université Mohammed 1, Oujda, Morocco }
\author[Mohammed Taous]{Mohammed Taous}
\author{Abdelkader Zekhnini}
\email{abdelmalekazizi@yahoo.fr}
\email{zekha1@yahoo.fr}
\address{Mohammed Taous: Département de Mathématiques, Faculté des Sciences et Techniques, Université Moulay Ismail, Errachidia, Morocco}
\email{taousm@hotmail.com}
\maketitle
\selectlanguage{english}
\begin{abstract}
 Let $d$ be a square-free integer, $\mathbf{k}=\QQ(\sqrt d,\,i)$ and $i=\sqrt{-1}$. Let $\mathbf{k}_1^{(2)}$ be the Hilbert $2$-class field of
$\mathbf{k}$, $\mathbf{k}_2^{(2)}$ be the Hilbert $2$-class field
of $\mathbf{k}_1^{(2)}$ and $G=\mathrm{Gal}(\mathbf{k}_2^{(2)}/\mathbf{k})$
be the Galois group of $\mathbf{k}_2^{(2)}/\mathbf{k}$.  Our goal is to give  necessary  and sufficient conditions to have  $G$    metacyclic in the case where $d=pq$, with $p$ and $q$ are primes  such that $p\equiv 1\pmod 8$ and $q\equiv 5\pmod 8$ or $p\equiv 1\pmod 8$ and $q\equiv 3\pmod 4$.
\end{abstract}
\section{Introduction}
 Let $k$ be an algebraic number field and let $Cl_2(k)$ denote its $2$-class group.  Denote by $k_2^{(1)}$ the Hilbert 2-class field of $k$ and  by $k_2^{(2)}$ its second  Hilbert 2-class field. Put $G=\mathrm{Gal}(k_2^{(2)}/k)$ and $G'$ its derived group, then its well known that $C/G'\simeq Cl_2(k)$. An important problem in Number Theory is to determine the structure of $G$, since the knowledge of $G$, its structure and its generators solve a lot of problems in number theory as capitulation problems, the finiteness or not of the towers  of number fields and the structures of the $2$-class groups of the unramified extensions of $k$ within $k_2^{(1)}$. In several times,  the knowledge of the rank of $G$ allows to know the structure of G.  In this paper, we give an example of this situation.

 Let $\mathbf{k}=\QQ(\sqrt{pq}, i)$, where $p$ and $q$ are two different primes,  then the genus field of $\mathbf{k}$ is $\mathbf{k}^*=\QQ(\sqrt{p}, \sqrt{q},\sqrt{-1} )$. According to \cite{AzTa09} $r_0$, the rank  of the $2$-class group of  $\mathbf{k}$,  is at most equal to  $3$. Moreover $r_0=3$ if and only  if  $p\equiv q\equiv 1\pmod 8$. Let $G=\mathrm{G}al(\mathbf{k}_2^{(2)}/\mathbf{k})$ be the Galois group of  $\mathbf{k}_2^{(2)}/\mathbf{k}$, where $\mathbf{k}_2^{(i+1)}$ is the Hilbert $2$-class field of $\mathbf{k}_2^{(i)}$, with $i=0$ or $1$ and $\mathbf{k}_2^{(0)}=\mathbf{k}$. The Artin Reciprocity implies that  $r_0=d(G)$, where $d(G)$ is the rank of  $G$.

In \cite{AzTa09}, the first and the second authors have shown that if $q=2$, then   $G$ is metacyclic non abelian if and only if  $p=x^2+32y^2$ and $x\not\equiv\pm 1\pmod 8$. In this paper,  we prove that if $p$ and $q$ are odd different primes, then $r$, the rank of $2$-class group of $\mathbf{k}^*$, helps to know in which case $G$ is metacyclic.
\section{The rank of the  $2$-class group of $\mathbf{k}^*$}
In what follows, we adopt the following notations: If $p\equiv 1\pmod 8$ is a prime, then
 $\left(\frac{2}{p}\right)_4$ will denote the rational biquadratic symbol which is equal to
 1 or -1, according as $2^{\frac{p-1}{4}}\equiv\pm 1\pmod p$. Moreover the
symbol $\left(\frac{p}{2}\right)_4$ is equal to $(-1)^{\frac{p-1}{8}}$. Let $k$ be a number field and $l$ be a prime; then $\mathfrak{l}_k$ will denote  a prime ideal of $k$ above $l$. We denote, also, by  $\left(\dfrac{x, y}{\mathfrak{l}_k}\right)$ (resp. $\displaystyle\left(\dfrac{x}{\mathfrak{l}_k}\right)$)  the Hilbert symbol (resp. the quadratic residue  symbol) for the prime $\mathfrak{l}_k$  applied to  $(x, y)$ (resp. $x$). A $2$-group $H$ is said of type  $(2^{n_1}, 2^{n_1}, ..., 2^{n_s})$  if it is  isomorphic to $\ZZ/2^{n_1}\times\ZZ/2^{n_2}\times...\ZZ/2^{n_s}$, where $n_i\in\NN$. For all number field  $k$, $h(k)$ will denote the $2$-class number of $k$. Finally, $r_0$ (resp. $r$) denotes the rank of the $2$-class group of $\mathbf{k}$ (resp. $\mathbf{k}^*$).
\begin{lem}\label{1}
If  $p\equiv 5\pmod 8$ and $q\equiv 3\pmod 4$, then $G$ is cyclic and $r=1$.
\end{lem}
\begin{proof}
 If $p\equiv 5\pmod 8$ and  $q\equiv 3\pmod 4$, then, according to \cite{McPaRa95}, the  $2$-class group of $\mathbf{k}$ is cyclic, so  $G$ is an abelian group of rank $1$. As  $\mathbf{k}^*$ is an unramified extension of  $\mathbf{k}$, then the  $2$-class group of $\mathbf{k}^*$ is also cyclic  and $r=1$.
\end{proof}
\begin{lem}\label{2}
If $p\equiv 1\pmod 8$ and $q\equiv 1\pmod 8$, then  $G$ is a non-metacyclic group .
\end{lem}
\begin{proof}
 If $p\equiv 1\pmod 8$ and  $q\equiv 1\pmod 8$, then, according to  \cite{McPaRa95}, the $2$-class group of  $\mathbf{k}$ is of rank $3$, which is the rank of  $G$. This yields that $G$ is not metacyclic, since the  metacyclic groups are of ranks  $\leq 2$.
\end{proof}
According to the two  Lemmas \ref{1} and \ref{2}, it is interesting to assume, in what follows, that  $p\equiv 1\pmod 8$ and $q\equiv 5\pmod 8$ or $p\equiv 1\pmod 8$ and  $q\equiv 3\pmod 4$. So    $r_0=2$ (see \cite{McPaRa95}). We continue with the following lemmas.
\begin{lem}[\cite{Lm00}]\label{5:004}
Let $p\equiv 1\pmod 8$ be a prime, then
\begin{center}
   $ \displaystyle\left(\dfrac{i}{\mathfrak{p}_{\QQ(i)}}\right)=1$\quad and \quad $\displaystyle\left(\dfrac{1+i}{\mathfrak{p}_{\QQ(i)}}\right)=\left(\dfrac{2}{p}\right)_4\left(\dfrac{p}{2}\right)_4$.
\end{center}
\end{lem}
\begin{lem}\label{3}
Let $F=\QQ(\sqrt{q},i)$ where $q\equiv 5\pmod 8$ and $\varepsilon_q$ be the fundamental unite of $\QQ(\sqrt{q})$, then
\begin{enumerate}[\rm(i)]
\item $\left(\dfrac{p, i}{\mathfrak{l}_{F}}\right)=1$ for all prime ideal  $\mathfrak{l}_{F}$ of $F$.
\item $\left(\dfrac{p, \varepsilon_q}{\mathfrak{l}_{F}}\right)=1$ for all odd prime ideal  $\mathfrak{l}_{F}\neq \mathfrak{p}_{F}$ of $F$.
\item $\left(\dfrac{p, \varepsilon_q}{\mathfrak{p}_{F}}\right)=\left(\dfrac{p, \varepsilon_q}{\mathfrak{2}_{F}}\right)=\left\{
                                                                 \begin{array}{ll}
                                                                   1, & \hbox{ if $\left(\dfrac{p}{q}\right)=-1$;} \\
                                                                   \left(\dfrac{p}{q}\right)_4\left(\dfrac{q}{p}\right)_4, & \hbox{ if $\left(\dfrac{p}{q}\right)=1$.}
                                                                 \end{array}
                                                               \right.
$
\end{enumerate}
\end{lem}
\begin{proof}
(i) Let $\mathfrak{l}_F$ be an odd prime ideal   of $F$.\\
 \indent If $\mathfrak{l}_F\neq \mathfrak{q}_F$, then $\mathfrak{l}_F$ is a prime ideal of $F$ unramified in  $F(\sqrt q)$ (see the proof of the following theorem), hence
\begin{align*}
    \left(\dfrac{p, i}{\mathfrak{l}_{F}}\right)=\displaystyle\left(\dfrac{p}{\mathfrak{l}_{F}}\right)^{v(i)}=1,  \text{  \cite[p. 205]{Gr03}}.
\end{align*}
\indent If $\mathfrak{l}_{F}=\mathfrak{p}_{F}$, the prime ideal of $F$ above  $p$, then
\begin{align*}
\left(\dfrac{p, i}{\mathfrak{p}_{F}}\right)&=\left(\dfrac{i, p}{\mathfrak{p}_{F}}\right)&  \\
&=\displaystyle\left(\dfrac{i}{\mathfrak{p}_{F}}\right) & \text{\cite[p. 205]{Gr03}}\\
&=\displaystyle\left(\dfrac{i}{\mathfrak{p}_{\QQ(i)}}\right) & \text{\cite[p. 205]{Gr03}}\\
&=1.&\text{(Lemma \ref{5:004}})
\end{align*}
\indent (ii) Same proof as in (i).\\
\indent (iii) the inertia degree of  $\mathfrak{p}_{F}$ is equal to $1$ in $F/\QQ(\sqrt{p})$, which implies that
\begin{align*}
\left(\dfrac{p,  \varepsilon_q}{\mathfrak{p}_{F}}\right)&=\left(\dfrac{ \varepsilon_q, p}{\mathfrak{p}_{F}}\right)& \\
&=\displaystyle\left(\dfrac{ \varepsilon_q}{\mathfrak{p}_{F}}\right) & \text{\cite[p. 205]{Gr03}}\\
&=\displaystyle\left(\dfrac{ \varepsilon_q}{\mathfrak{p}_{\QQ(\sqrt{q})}}\right) & \text{\cite[p. 205]{Gr03}}
\end{align*}
Suppose that $\left(\dfrac{p}{q}\right)=1$, then
\begin{align*}
&= \left(\dfrac{p}{q}\right)_4\left(\dfrac{q}{p}\right)_4.&\text{\cite[p. 101]{Sc34}}
\end{align*}
Suppose that $\left(\dfrac{p}{q}\right)=-1$. With a same argument as above, we get:
\begin{align*}
\left(\dfrac{p,  \varepsilon_q}{\mathfrak{p}_{F}}\right)=\displaystyle\left(\dfrac{\mathcal{N}_{\QQ(\sqrt{q})/\QQ)}( \varepsilon_q)}{p}\right)
= \displaystyle\left(\dfrac{-1}{p}\right)=1, \text{  \cite[p. 205]{Gr03}}.
\end{align*}
The product formula for the Hilbert symbol implies that  $\left(\dfrac{p, \varepsilon_p}{\mathfrak{p}_{F}}\right)=\left(\dfrac{p, \varepsilon_p}{\mathfrak{2}_{F}}\right)$.
\end{proof}
\begin{lem}\label{4}
 Let $F=\QQ(\sqrt{q},i)$ where $q\equiv 3\pmod 4$  and $\varepsilon_q$ be the fundamental unite of $\QQ(\sqrt{q})$. Then
\begin{enumerate}[\rm(i)]
\item $\left(\dfrac{p, i}{\mathfrak{l}_{F}}\right)=1$ for all prime ideal $\mathfrak{l}_{F}$ of $F$.
\item $\left(\dfrac{p, \sqrt{i\varepsilon_q}}{\mathfrak{l}_{F}}\right)=1$ for all odd prime ideal  $\mathfrak{l}_{F}\neq \mathfrak{p}_{F}$ of $F$.
\item $\left(\dfrac{p, \sqrt{i\varepsilon_q}}{\mathfrak{p}_{F}}\right)=\left(\dfrac{p, \sqrt{i\varepsilon_q}}{\mathfrak{2}_{F}}\right)=\left\{
                                                                 \begin{array}{l}
                                                                   1  \hbox{, if $\left(\dfrac{p}{q}\right)=-1$;} \\
                                                                   \left(\dfrac{2}{p}\right)_4\left(\dfrac{p}{2}\right)_4\left(\dfrac{\sqrt{2\varepsilon_q}}{\mathfrak{p}_{\QQ(\sqrt{q})}}\right) \hbox{, if $\left(\dfrac{p}{q}\right)=1$.}
                                                                 \end{array}
                                                               \right.
$
\end{enumerate}
\end{lem}
\begin{proof}By a similar approach of previous lemma, we  get (i) and (ii). For (iii), remark that $2\varepsilon_q$  is a square in $\QQ(\sqrt{q}) $ (see \cite{Az-00}), then $i\varepsilon_q$ is a square in $F$, since $2 \sqrt{i\varepsilon_q}=(1+i) \sqrt{2\varepsilon_q}$. As the Hilbert symbol is a bilinear map with values in $\{+1, -1\}$ and $2i=(1+i)^2$, so
\begin{align*}
\left(\dfrac{p, \sqrt{i\varepsilon_q}}{\mathfrak{p}_{F}}\right)&=\left(\dfrac{p, 2}{\mathfrak{p}_{F}}\right)\left(\dfrac{p, 1+i}{\mathfrak{p}_{F}}\right)\left(\dfrac{p, \sqrt{2\varepsilon_q}}{\mathfrak{p}_{F}}\right)&\\
&=\left(\dfrac{p, i}{\mathfrak{p}_{F}}\right)\left(\dfrac{p, 1+i}{\mathfrak{p}_{F}}\right)\left(\dfrac{p, \sqrt{2\varepsilon_q}}{\mathfrak{p}_{F}}\right)\\
&=\left(\dfrac{1+i}{\mathfrak{p}_{F}}\right)\left(\dfrac{\sqrt{2\varepsilon_q}}{\mathfrak{p}_{F}}\right)\\
&=\left(\dfrac{1+i}{\mathfrak{p}_{\QQ(i)}}\right)\left(\dfrac{\sqrt{2\varepsilon_q}}{\mathfrak{p}_{\QQ(\sqrt{q})}}\right)\text{\cite[p. 205]{Gr03}}\\
&=\left(\dfrac{2}{p}\right)_4\left(\dfrac{p}{2}\right)_4\left(\dfrac{\sqrt{2\varepsilon_q}}{\mathfrak{p}_{\QQ(\sqrt{q})}}\right).
\end{align*}
\end{proof}
\begin{thm}\label{woroujda:1}
Let $p$ and $q$ be primes  as  above and $r$ be the rank of the $2$-class group of $\QQ(\sqrt{q}, \sqrt{p}, i)$.
\begin{enumerate}[\rm\indent(1)]
\item If $q\equiv 5\pmod 8$, then $$r=\left\{
                             \begin{array}{ll}
                               1, & \hbox{if $\left(\dfrac{p}{q}\right)=-1$;} \\
                               2, & \hbox{if $\left(\dfrac{p}{q}\right)=1$ and $\left(\dfrac{p}{q}\right)_4=-\left(\dfrac{q}{p}\right)_4$;} \\
                               3, & \hbox{if $\left(\dfrac{p}{q}\right)=1$ and $\left(\dfrac{p}{q}\right)_4=\left(\dfrac{q}{p}\right)_4$.}
                             \end{array}
                           \right.$$
\item If $q\equiv 3\pmod 4$, then, by putting  $\eta=\left(\dfrac{\sqrt{2\varepsilon_q}}{\mathfrak{p}_{\QQ(\sqrt{q})}}\right)$ if $\left(\dfrac{p}{q}\right)=1$,  we obtain  $$r=\left\{
                             \begin{array}{ll}
                               1, & \hbox{if $\left(\dfrac{p}{q}\right)=-1$;} \\
                               2, & \hbox{if $\left(\dfrac{p}{q}\right)=1$ and $\left(\dfrac{2}{p}\right)_4=-\left(\dfrac{p}{2}\right)_4\eta$;} \\
                               3, & \hbox{if $\left(\dfrac{p}{q}\right)=1$ and $\left(\dfrac{2}{p}\right)_4=\left(\dfrac{p}{2}\right)_4\eta$.}
                             \end{array}
                           \right.$$
\end{enumerate}
\end{thm}
\begin{proof}
Let $F$ denote the field  $\QQ(\sqrt q,i)$ defined above  and $\varepsilon_q$ be the fundamental unit of the $\QQ(\sqrt{q})$. According to \cite{Az99:1},   the unit group  of $F$ is equal to
$$
\left\{
  \begin{array}{ll}
    \langle i, \varepsilon_q\rangle, & \hbox{if $q\equiv 5\pmod 8$;} \\
     \langle i, \sqrt{i\varepsilon_q}\rangle, & \hbox{if $q\equiv 3\pmod 4$.}
  \end{array}
\right.
$$
\indent As the class number of the $F$ is odd, then by the ambiguous class number formula (see \cite{Ch33}),  we have :
\begin{equation*}
r= t-e-1,
\end{equation*}
where $t$ is the number of primes of $F$ that ramify in  $\mathbf{k}^*/F$ ($\mathbf{k}^*=\QQ(\sqrt{q}, \sqrt{p}, i)$ is the genus field of
$\mathbf{k}=\QQ(\sqrt{pq}, i)$) and $e$ is determined by  $2^e=[E_F:E_F\cap N_{\mathbf{k}^*/F}((\mathbf{k}^*)^\times)]$. The following diagram helps us to calculate the number $t$.
 \begin{figure}[H]
$$
 \xymatrix@R=0.22cm@C=0.4cm{&\QQ(\sqrt q)\ar[drr]\\
&&& F=\QQ(\sqrt q, i)\ar[drr]\\
\QQ \ar[ddr]\ar[uur] \ar[r] &\QQ(i)\ar[urr]\ar[drr]&&&& \mathbf{k}^*=\QQ(\sqrt q, \sqrt{p}, i)\\
  &&& \mathbf{k}=\QQ(\sqrt{pq},i) \ar[urr]\\
& \QQ(\sqrt{pq})\ar[urr]  }
$$
\caption{}
\end{figure}
\noindent Let $l$ be a prime. Since the extension $\mathbf{k}/\mathbf{k}^*$  is unramified, then
$$e(\mathfrak{l}_{F}/l).e(\mathfrak{l}_{\mathbf{k}^*}/\mathfrak{l}_{F})=e(\mathfrak{l}_{\mathbf{k}}/l).$$
As $$e(\mathfrak{l}_{F}/l)=\left\{
                                       \begin{array}{ll}
                                         2 & \hbox{if $l=q$ or $2$,} \\
                                         1 & \hbox{otherwise,}
                                       \end{array}
                                     \right.$$
 and  $$e(\mathfrak{l}_{\mathbf{k}}/l)=\left\{
                                       \begin{array}{ll}
                                         2 & \hbox{if $l=p$, $q$ or $2$,} \\
                                         1 & \hbox{otherwise,}
                                       \end{array}
                                     \right.$$
so it is easy to see that $$e(\mathfrak{l}_{\mathbf{k}^*}/\mathfrak{l}_{F})=\left\{
                                       \begin{array}{ll}
                                         2 & \hbox{if $l=p$,} \\
                                         1 & \hbox{otherwise.}
                                       \end{array}
                                     \right.$$
Similarly, we find that $$f(\mathfrak{l}_{\mathbf{k}^*}/\mathfrak{l}_{F})=\left\{
                                       \begin{array}{ll}
                                         1 & \hbox{if $\left(\dfrac{p}{q}\right)=1$,} \\
                                         2 & \hbox{if $\left(\dfrac{p}{q}\right)=-1$.}
                                       \end{array}
                                     \right.$$
Therefore $$t=\left\{
                                       \begin{array}{ll}
                                         4 & \hbox{if $\left(\dfrac{p}{q}\right)=1$,} \\
                                         2 & \hbox{if $\left(\dfrac{p}{q}\right)=-1$.}
                                       \end{array}
                                     \right.$$
Finally $$r=\left\{
                                       \begin{array}{ll}
                                         3-e & \hbox{if $\left(\dfrac{p}{q}\right)=1$,} \\
                                         1-e & \hbox{if $\left(\dfrac{p}{q}\right)=-1$.}
                                       \end{array}
                                     \right.$$
The Hasse norm theorem (see e.g. \cite[theorem 6.2, p. 179]{Gr03}) implies that a unit $\varepsilon $ of $F$ is a norm of an element of  $F(\sqrt{p})=\mathbf{k}^*$ if and only if $\left(\dfrac{p, \varepsilon}{\mathfrak{p}_F}\right)=1$, for all $\mathfrak{p}_F\neq \mathfrak{2}_F$ prime ideal of $F$.\\
 If $q\equiv 3\pmod 4$, we conclude
 thanks to the previous lemma that $$e=\left\{
                                         \begin{array}{ll}
                                           0, & \hbox{if $\left(\dfrac{p}{q}\right)=-1$;} \\
                                           1, & \hbox{if $\left(\dfrac{p}{q}\right)=1$ and $\left(\dfrac{2}{p}\right)_4=-\left(\dfrac{p}{2}\right)_4\eta$;} \\
                                           0, & \hbox{if $\left(\dfrac{p}{q}\right)=1$ and $\left(\dfrac{2}{p}\right)_4=\left(\dfrac{p}{2}\right)_4\eta$.}
                                         \end{array}
                                       \right.$$
 This completes the proof of our theorem.
\end{proof}

\section{Application}
Let $G=\mathrm{Gal}(k_2^{(2)}/k)$ be the Galois group of the extension $\mathbf{k}_2^{(2)}/\mathbf{k}$, where  $\mathbf{k}=\QQ(\sqrt{pq}, i)$ and $p$,  $q$ are distinct primes such that  $p\equiv 1\pmod 8$ and $q\equiv 5\pmod 8$ or $p\equiv 1\pmod 8$ and $q\equiv 3\pmod 4$. In this section, we give an application of the previous theorem and  we characterize the group $G$. In particular,  we will find results about $G$ given by Azizi in \cite{Az99:3} and \cite{Az02}.
\begin{thm}\label{}
Let $p$ and $q$ be different primes  defined as above, $\mathbf{k}=\QQ(\sqrt{pq}, i)$, $r$ be the rank  of the $2$-class group of $\QQ(\sqrt{q}, \sqrt{p}, i)$ and $G=\mathrm{G}al(\mathbf{k}_2^{(2)}/\mathbf{k})$. Then $G$ is nonmetacyclic  if and only if $r=3$.
\end{thm}
\begin{proof}
Since $p\equiv 1\pmod 8$, then there exist two integers $x$ and $y$
such that $p=x^2+16y^2$. Put $\pi_1=x+4yi$, $\pi_2=x-4yi$,
$\mathbf{k}_1=\mathbf{k}(\sqrt{\pi_1})$ and $\mathbf{k}_2=\mathbf{k}(\sqrt{\pi_2})$.
As $\pi_1$ and $\pi_2$ are ramified in $\mathbf{k}/\QQ(i)$, then the ideals generated by $\pi_1$ and $\pi_2$
are squares of ideals of $\mathbf{k}$. Note that $x$ is odd, thus  $x\equiv  \pm 1\equiv i^2\pmod 4$, then the two equations $\pi_i\equiv\xi^2 $ are solvable in $\mathbf{k}$. We conclude that  the two extensions $\mathbf{k}(\sqrt{\pi_1})/\mathbf{k}$ and $\mathbf{k}(\sqrt{\pi_2})/\mathbf{k}$ are unramified. It is clear that $\mathbf{k}_i\neq\mathbf{k}^*$. Since $r_0=d(G)=2$, then $\mathbf{k}_1$, $\mathbf{k}_2$ and $\mathbf{k}^*$ are  precisely the three unramified quadratic extensions of $\mathbf{k}$. Put $\mathrm{H}_i=\mathrm{G}al(\mathbf{k}_2^{(2)}/\mathbf{k}_i)$ where $i=1, 2$ and $\mathrm{M}=\mathrm{G}al(\mathbf{k}_2^{(2)}/\mathbf{k^*})$. These three subgroups are the maximal subgroups of $G$. As $d(G)=2$ and $\mathbf{k}_1$ is isomorphic to $\mathbf{k}_2$, then according to \cite{BeJa09},  $G$ is nonmetacyclic  if and only if $d(M)=r=3$.
\end{proof}

\begin{lem}\label{woroujda:2}
 Let $k$ be an algebraic number field and  $G=\mathrm{Gal}(k_2^{(2)}/k)$. Let $L$ be an extension of $k$ such that $M=\mathrm{Gal}(k_2^{(2)}/L)$ is a cyclic subgroup of $G$ of index $2$. If  four ideal classes of $k$ capitulate in $L$, then $G$ is abelian or  dihedral group.
\end{lem}
\begin{proof}
we say that an ideal class of $k$ capitulates in $L$ if it is in the kernel of the homomorphism
$j:  Cl_2(k) \longrightarrow  Cl_2(L)$ induced by extension of ideals from $k$
to $L$. Furthermore, the homomorphism $j$ corresponds, by the Artin reciprocity law to the group theoretical transfer $V: G/G' \longrightarrow M$ (For further information on $V$, see for example \cite{Mi89}). Since $M$ is a cyclic subgroup of $G$ of index $2$, then, according to \cite{KS04}, $G$ is isomorphic to one of the following groups:
\begin{enumerate}[\indent\rm(1)]
  \item Cyclic $2$-group or $2$-group of type $(2^n, 2^m)$.
  \item The dihedral group.
  \item The quaternion group.
  \item The semidihedral group.
  \item The modular $2$-group.
\end{enumerate}
If $G$ is one of the first four groups and four ideal classes of $k$ capitulate in $L$, then H. Kisilevsky has shown in \cite{Ki76} that $G$ is abelian or  dihedral group. Next if $G$ is a modular $2$-group, then $G=\langle x, y : x^{2^{n-1}}=y^2, y^{-1}xy=x^{1+2^{n-2}}\rangle$ and $M=\langle x\rangle$. An elementary calculation shows that $\ker(V)=\{G', yG'\}$, i.e. two ideal classes of $k$ capitulate in $L$.
\end{proof}

\begin{thm}\label{}
Let $p$ and $q$ be different primes and  $\eta$ be the number  defined in Theorem $\ref{woroujda:1}$. Put $\mathbf{k}=\QQ(\sqrt{pq}, i)$ and $G=\mathrm{G}al(\mathbf{k}_2^{(2)}/\mathbf{k})$.
\begin{enumerate}[\indent\rm(1)]
\item If $q\equiv 5\pmod 8$ and $\left(\dfrac{p}{q}\right)=-1$, then $G$ is  dihedral.
\item If $q\equiv 5\pmod 8$ and $\left(\dfrac{p}{q}\right)=1$ and $\left(\dfrac{p}{q}\right)_4=-\left(\dfrac{q}{p}\right)_4$, then $G$ is a nonabelian metacyclic group with $G/G'$ is of type $(2, 4)$.
\item If $q\equiv 3\pmod 8$ and $\left(\dfrac{p}{q}\right)=-1$, then $G$ is  abelian or dihedral.
\item If $q\equiv 3\pmod 4$ and $\left(\dfrac{p}{q}\right)=1$ and $\left(\dfrac{2}{p}\right)_4=-\left(\dfrac{p}{2}\right)_4\eta$, then $G$ is a metacyclic group.
\end{enumerate}
\end{thm}
\begin{proof}
Proceeding as in the proof of  Theorem 7 of \cite{Az99:2}, we get  that if $q\equiv 5\pmod 8$, then $h(\mathbf{k}^*)=\dfrac{h(-p)h(\mathbf{k})}{4}$,  where $h(-p)$ denotes the $2$-class number of $\QQ(\sqrt{-p})$. Since $r_0$, the rank of the $2$-class group of $k$, is equal to $2$, then $G$ is abelian if and only if si $h(\mathbf{k}^*)=\dfrac{h(\mathbf{k})}{2}$ (see \cite{BeLeSn98}), this is equivalent to $h(-p)=2$. Which is impossible since $p\equiv 1\pmod 8$.

(1) If $q\equiv 5\pmod 8$ and $\left(\dfrac{p}{q}\right)=-1$, then $G$ is nonabelian. According to Theorem \ref{woroujda:1} the rank of the $2$-class group of $\mathbf{k}^*$ is $r=1$, thus $\mathrm{M}=\mathrm{G}al(\mathbf{k}_2^{(2)}/\mathbf{k^*})$ is cyclic. On the other hand,  Azizi in \cite{Az99:2} has shown, in this situation,  that there are four ideal classes of $\mathbf{k}$ capitulate in $\mathbf{k^*}$, hence the  Lemma \ref{woroujda:2} implies that $G$ is a dihedral group.

(2) If $q\equiv 5\pmod 8$,  $\left(\dfrac{p}{q}\right)=1$ and  $\left(\dfrac{p}{q}\right)_4=-\left(\dfrac{q}{p}\right)_4$, then $G$ is nonabelian and the  $2$-class group of $\mathbf{k}$ is of type $(2, 4)$ (see \cite{AzTa08}). Moreover Theorem \ref{woroujda:1}  yields that  $r=2$, then, according to the previous theorem, we have $G$ is metacyclic.

(3) and (4) are proved similarly.
\end{proof}

\end{document}